\documentclass[11pt]{article}
\usepackage{amsmath,amssymb,amsthm}
\usepackage[margin=1in]{geometry}
\usepackage[hidelinks]{hyperref}
\newtheorem{theorem}{Theorem}
\newtheorem{lemma}{Lemma}
\newcommand{\ex}{\operatorname{ex}}
\title{On the Tur\'an Number of the Cartesian Product of Two Stars}
\author{
    Quan Sun\thanks{School of Mathematical Sciences, East China Normal University, 500 Dongchuan Road, Shanghai 200240, P.R.  China.
        Email: 52280155016@stu.ecnu.edu.cn
    }
    \qquad
    Long-Tu Yuan\thanks{School of Mathematical Sciences and Shanghai Key Laboratory of PMMP, East China Normal University, 500 Dongchuan Road, Shanghai 200240, P.R.  China.
        Email: ltyuan@math.ecnu.edu.cn.
        Supported in part by National Natural Science Foundation of China grant 12271169 and Science and Technology Commission
        of Shanghai Municipality (No. 22DZ2229014).
    }
}
\date{}
\begin{document}
\maketitle

\begin{abstract}
We study the Tur\'an number of the Cartesian product of two stars,
a family whose order of magnitude is known. For most pairs of star
sizes, we determine the leading coefficient exactly. In the remaining
congruence class, we give explicit lower and upper bounds for the
coefficient at the leading scale. Our lower bounds use blowups of
polarity graphs, and our upper bounds use structural estimates on
neighbourhoods.
\end{abstract}

\section{Introduction}

For a graph $H$, the \emph{Tur\'an number} $\ex(n,H)$ is the maximum number
of edges in a graph $G$ on $n$ vertices that does not contain $H$ as a subgraph.

For a bipartite graph $H$ containing a cycle, determining an exponent
$\alpha$ for which $\ex(n,H)=\Theta(n^\alpha)$ and proving a formula
$\ex(n,H)=(c_H+o(1))n^\alpha$ with an explicit $c_H>0$ are
distinct problems. The latter is an important and often substantially
harder refinement; see \cite[Remark~3.3]{FurediSimonovits}.
An estimate of the order of magnitude determines the growth rate but allows
a multiplicative gap between lower and upper bounds. The leading constant
specifies the limiting value of $\ex(n,H)/n^\alpha$ and, to first order,
the number of edges that forces a copy of $H$. It also distinguishes
constructions whose edge counts have the same exponent.

The difficulty is that a construction and a universal upper bound
must agree to first order. Estimates sufficient to establish
$\Theta(n^\alpha)$ need not retain this precision, and even the
existence of a limit for $\ex(n,H)/n^\alpha$ is a separate question.
Exact leading constants are known for comparatively few families.
A classical example is $C_4$,
with coefficient $1/2$ at exponent $3/2$ \cite{ERS,KST}.
For each fixed integer $t\geq2$, F\"uredi \cite{FurediK2t} proved
$\ex(n,K_{2,t})=(\sqrt{t-1}/2+o(1))n^{3/2}$.
Brown constructed graphs that do not contain $K_{3,3}$ as a subgraph
and have $(1/2-o(1))n^{5/3}$
edges, and F\"uredi proved a matching upper bound, giving
$\ex(n,K_{3,3})=(1/2+o(1))n^{5/3}$
\cite{BrownThomsen,FurediZarankiewicz}. The hexagon
illustrates the additional difficulty of the coefficient problem:
at the known scale $n^{4/3}$, F\"uredi, Naor and Verstra\"ete
\cite{FurediNaorVerstraete} disproved the conjectured coefficient
$1/2$ and obtained more precise, but nonmatching, coefficient bounds.

We study the Cartesian product of two stars. For graphs $G$ and $H$,
the \emph{Cartesian product} $G\square H$ has vertex set $V(G)\times V(H)$;
two vertices $(u,v)$ and $(u',v')$ are adjacent if $u=u'$ and
$vv'\in E(H)$, or if $v=v'$ and $uu'\in E(G)$. If both factors are
bipartite, their product is bipartite: choose a proper colouring of
each factor with colours $0,1$, and colour $(u,v)$ by the sum of
the two colours modulo $2$. Brada\v{c}, Janzer, Sudakov
and Tomon \cite{Grid} determined the order of magnitude for the
Cartesian product of a tree and a path, each with at least one edge,
including grids. They also observed in their concluding remarks
that their argument extends to the product of two trees, each with
at least one edge,
when one factor is a star. Thus the order $\Theta(n^{3/2})$ for
the products studied here was already known. Gao, Janzer, Liu and
Xu \cite{GeometricShapes} improved the coefficient in the $n^{3/2}$ upper bound for
grids and determined the order of magnitude for the product of an
even cycle of length at least eight and an edge. Wang and Yang
\cite{TreeProduct} determined the order of magnitude for the product
of any two trees with at least one edge each. Zhao, Cheng, Chi,
Gy\H{o}ri, Tompkins and Wang \cite{ZhaoEtAl} obtained upper and lower
bounds at the $n^{3/2}$ scale for the product of a star and an edge.
For the product of two stars, the known exponent therefore leaves
the coefficient problem to be addressed. We determine the exact
leading constant outside one congruence class and give explicit
coefficient bounds within that class.
Our upper bound proofs draw on methods of Gao et al.\ \cite{GeometricShapes}
and Zhao et al.\ \cite{ZhaoEtAl}; the specific connections are described
in Section~3.

We write $S_t=K_{1,t}$ for the star with $t$ leaves. The star sizes
$a,b$ are fixed while the host order $n$ tends to infinity.
Throughout the paper, write $D=\lfloor a(a+b)/(a+1)\rfloor$ and
$F_{a,b}=S_a\square S_b$. The following are our main results.

\begin{theorem}\label{thm:general}
Fix integers $2\leq a\leq b$ with $b\not\equiv2\pmod{a+1}$. Then
\[
 \ex(n,F_{a,b})=\left(\frac{\sqrt D}{2}+o(1)\right)n^{3/2}.
\]
\end{theorem}

\begin{theorem}\label{thm:exceptional}
Fix integers $2\leq a<b$ with $b\equiv2\pmod{a+1}$. Then
\[
 \left(\frac{\sqrt{D+1/10}}2-o(1)\right)n^{3/2}
 \leq\ex(n,F_{a,b})
 \leq\left(\frac{\sqrt{D+1/2}}2+o(1)\right)n^{3/2}.
\]
\end{theorem}

The paper is organized as follows. Section~2 gives notation,
preliminary facts and both constructions for the lower bounds.
Section~3 proves the upper bounds and
completes the proofs of Theorems~\ref{thm:general} and
\ref{thm:exceptional}.

\section{Preliminaries and lower bounds}

All graphs are finite, undirected and simple. An embedding is an
injective map on vertices that preserves adjacency; the copy it
defines need not be induced. We write $V(G)$,
$E(G)$, $e(G)$, $N_G(v)$ and $d_G(v)$ for the \emph{vertex set},
\emph{edge set}, \emph{number of edges}, \emph{neighbourhood} and
\emph{degree}. For a nonempty vertex set $U$, its
\emph{common neighbourhood} is
$N_G(U)=\bigcap_{u\in U}N_G(u)$.
When the ambient graph is clear, we write $N(v)$, $d(v)$ and $N(U)$
in place of $N_G(v)$, $d_G(v)$ and $N_G(U)$, respectively.
The maximum degree is $\Delta(G)=\max_{v\in V(G)}d_G(v)$.
For $U\subseteq V(G)$, the \emph{induced subgraph} $G[U]$ has
vertex set $U$ and all edges of $G$ with both endpoints in $U$.
The notation $C_4$ denotes a \emph{cycle of length four}. Unless another limiting
variable is explicitly indicated, asymptotic statements are taken
as $n\to\infty$ with $a,b$ fixed. Error estimates are uniform
over the host graphs under consideration; their constants may depend
on $a,b$ and on any other explicitly fixed parameters.
For a positive integer $m$, write
$[m]=\{1,\ldots,m\}$.

Write the centres of $S_a,S_b$ as $c_a,c_b$, and their leaves as
$p_1,\ldots,p_a$ and $q_1,\ldots,q_b$, respectively. In
$F_{a,b}=S_a\square S_b$, let $o=(c_a,c_b)$ be the \emph{central vertex},
$x_i=(p_i,c_b)$ the \emph{row vertices}, $y_j=(c_a,q_j)$ the
\emph{column vertices}, and $z_{ij}=(p_i,q_j)$ the
\emph{cell vertices}. These names
refer only to their positions in the product. The edges of
$F_{a,b}$ are $ox_i$, $oy_j$, $x_iz_{ij}$ and $y_jz_{ij}$.
Each pair $(i,j)$ gives a cycle of length four
$ox_iz_{ij}y_jo$ with opposite pairs $\{o,z_{ij}\}$ and
$\{x_i,y_j\}$.

For both constructions proving lower bounds, we use the following dense
\emph{polarity graph} $\mathcal P_\ell$, which does not contain
$C_4$ as a subgraph, as the underlying graph.
For a prime power $\ell$, let $\mathbb F_\ell$ be the field with
$\ell$ elements. The \emph{points} of the \emph{projective plane}
$\mathrm{PG}(2,\ell)$ are the subspaces of dimension one
$[u]=\mathbb F_\ell u$ of $\mathbb F_\ell^3$, where $u\ne0$.
There are $N=(\ell^3-1)/(\ell-1)=\ell^2+\ell+1$ such points.
For $u,v\in\mathbb F_\ell^3$, define
\[
 B(u,v)=u_1v_1+u_2v_2+u_3v_3,
\]
with all arithmetic in $\mathbb F_\ell$. This is a nondegenerate
symmetric bilinear form. Define $\mathcal P_\ell$ by joining
distinct $[u]$ and $[v]$ exactly when $B(u,v)=0$.
This condition is independent of the chosen nonzero representatives.
This is the construction of \cite{ERS}.
Let $A$ be the symmetric $N\times N$ matrix with entries in $\{0,1\}$, regarded
as a matrix over $\mathbb R$, defined by
$A_{[u],[v]}=1$ if $B(u,v)=0$, and
$A_{[u],[v]}=0$ otherwise. Unlike the adjacency matrix of
$\mathcal P_\ell$, the matrix $A$ retains any diagonal entry for which
$B(u,u)=0$.
Every row of $A$ contains $(\ell^2-1)/(\ell-1)=\ell+1$ ones:
for nonzero $u$, the equation $B(u,w)=0$ defines a subspace of dimension two. Any two distinct rows have a one in exactly one common column. For
distinct projective points $[u],[v]$, the two independent linear
equations $B(u,w)=B(v,w)=0$ have a solution space of dimension one. Since $A$ is symmetric, ordinary matrix multiplication gives
\[
 A^2=\ell I+J,
\]
where $I$ is the \emph{identity matrix} and $J$ is the \emph{all ones matrix},
both of order $N$. Setting all diagonal entries of $A$ to zero
gives the adjacency matrix of $\mathcal P_\ell$.
Consequently $\mathcal P_\ell$ does not contain $C_4$ as a subgraph,
has maximum degree at most $\ell+1$, and has
$(\ell+1)N/2+O(N)=(1/2+o(1))N^{3/2}$ edges as $\ell\to\infty$.

\subsection{Lower bounds in Theorems~\ref{thm:general} and~\ref{thm:exceptional}}

The constructions below adapt the blowups of Zhao et al.\
\cite[Section~2.2]{ZhaoEtAl}, including the use of two consecutive
fibre sizes and the deletion of perfect matchings between larger fibres.
Uniform blowups of graphs containing no $C_4$ as a subgraph were used
earlier by F\"uredi \cite[p.~76]{FurediTuranType} and
Brada\v{c} et al.\ \cite[Section~3]{Grid}.

\begin{proof}[Proof of the lower bounds]
Fix integers $2\leq a\leq b$.

For each $p\in V(\mathcal P_\ell)$, take an independent set $W_p$
of size $D$ or $D+1$, called the fibre over $p$, with all
fibres pairwise disjoint. Define $G_\ell$ on their union as follows.
If $pq\in E(\mathcal P_\ell)$, include all edges between $W_p$
and $W_q$, except that when both fibres have size $D+1$, delete
an arbitrary perfect matching between them. If $p\ne q$ and
$pq\notin E(\mathcal P_\ell)$, include no edges between $W_p$
and $W_q$. The projection is the map
$\pi:V(G_\ell)\to V(\mathcal P_\ell)$ with $\pi(w)=p$ for
$w\in W_p$.
For Theorem~\ref{thm:general}, all fibres will have size $D$;
this is the uniform blowup of $\mathcal P_\ell$ with fibres of size $D$.
For Theorem~\ref{thm:exceptional}, we allow both sizes.

Suppose that $F_{a,b}$ embeds in $G_\ell$, and identify its vertices
with their images. Let $M=D$ for the construction in
Theorem~\ref{thm:general} and $M=D+1$ for the construction in
Theorem~\ref{thm:exceptional}; thus every fibre has at most $M$
vertices. For each $(i,j)\in[a]\times[b]$, the
cycle of length four $ox_iz_{ij}y_jo$ projects to a closed walk of length
four in $\mathcal P_\ell$. Consecutive images are distinct,
since the fibres are independent. If all four images were
distinct, they would form a $C_4$ in $\mathcal P_\ell$.
Therefore $\pi(o)=\pi(z_{ij})$ or $\pi(x_i)=\pi(y_j)$.

For each $(i,j)$ choose one of these equalities that holds;
if both hold, choose either one. Let $\mathcal A$ consist of the
pairs $(i,j)$ for which $\pi(o)=\pi(z_{ij})$ was chosen. The distinct vertices
$o$ and $z_{ij}$ with $(i,j)\in\mathcal A$ all lie in
$W_{\pi(o)}$, so $|\mathcal A|+1\leq M$.
Let $H$ be the auxiliary bipartite graph with parts
$\{x_1,\ldots,x_a\}$ and $\{y_1,\ldots,y_b\}$, where
$x_iy_j\in E(H)$ exactly when $(i,j)\notin\mathcal A$.
Thus $|\mathcal A|=ab-e(H)$.
For every edge $x_iy_j$ of $H$, the chosen equality is
$\pi(x_i)=\pi(y_j)$. The map $\pi$ is therefore constant on the
vertex set of each connected component of $H$, so every component
has at most $M$ vertices. Different components may lie in the same fibre.

If $H$ has no edges, then $|\mathcal A|=ab$. Otherwise choose
a component $C$ with the largest number $r\geq1$ of row vertices,
and let $m$ be its number of column vertices. Define
$y=\min\{b,M-r\}$, so $m\leq y$.
For each column vertex $y_j\in V(C)$, at least $a-r$ row
vertices are not adjacent to $y_j$ in $H$.
A column vertex outside $C$ belongs to a component with at most
$\min\{r,a-r\}$ row vertices, by the maximality of $r$ and the
number of row vertices outside $C$.
For each such $y_j$, at least $\max\{r,a-r\}$ row vertices
are not adjacent to $y_j$ in $H$.
Since $|\mathcal A|=\sum_{j=1}^b(a-d_H(y_j))$ and
$a-r\leq\max\{r,a-r\}$, we obtain
\[
 |\mathcal A|
 \geq m(a-r)+(b-m)\max\{r,a-r\}
 \geq y(a-r)+(b-y)\max\{r,a-r\}.
\]
The second inequality follows by replacing $m$ by its upper
bound $y$ in an expression that is nonincreasing in $m$.

For Theorem~\ref{thm:general}, take all fibres of size $D$, so
$M=D$ and $|\mathcal A|\leq D-1$.
The following argument applies to every $2\leq a\leq b$.
Write $b-2=(a+1)u+s$, where $u\geq0$ and $0\leq s\leq a$
are integers. The definition of $D$ gives $D=a+b-2-u$.
If $H$ has no edges, then $|\mathcal A|=ab\geq a+b-2\geq D$,
a contradiction. Otherwise use the preceding bound with
$y=\min\{b,D-r\}$.

If $D-r\geq b$, then $y=b$ and $r\leq a-2-u$.
Consequently $|\mathcal A|\geq b(a-r)\geq b(u+2)\geq D$,
where $b(u+2)\geq2b\geq a+b-2\geq D$.
If $D-r<b$ and $r<a/2$, then
$|\mathcal A|\geq b(a-r)\geq b\lceil a/2\rceil
\geq a+b-2\geq D$; the penultimate inequality follows from
$b(\lceil a/2\rceil-1)\geq a-2$ for $a\geq2$ and $b\geq a$.

In the remaining case $D-r<b$ and $r\geq a/2$, let $k=a-r$.
Then $0\leq k\leq u+1$, $y=D-r$, and subtracting $D$ from
the lower bound $y(a-r)+(b-y)r$ gives
$a-s+k((a-1)u+s-a-2)+2k^2$.
For $k=0$ this equals $a-s\geq0$, and for $k=1$ it equals
$(a-1)u\geq0$. For $k\geq2$, use $u\geq k-1$ and
$(k-1)s\geq0$ to obtain the lower bound
$a+(a+1)k^2-(2a+1)k$. This quadratic is increasing for
$k\geq2$ and equals $a+2>0$ at $k=2$.
Thus $|\mathcal A|\geq D$ in every case, contradicting
$|\mathcal A|\leq D-1$.

Suppose $a<b$ and $b\equiv2\pmod{a+1}$. Write
$b=(a+1)t+2$ with an integer $t\geq1$; then
$D=a(t+1)\geq4$.
Allow arbitrary fibres of sizes $D$ and $D+1$, so $M=D+1$
and $|\mathcal A|\leq D$. If $H$ has no edges, then
$|\mathcal A|=ab>D$, since $b>t+1$.
Otherwise the common bound applies with
$y=\min\{b,D+1-r\}$.

If $D+1-r\geq b$, then $a-r\geq t+1$, so
$|\mathcal A|\geq b(a-r)\geq b(t+1)>a(t+1)=D$.
If $D+1-r<b$ and $r<a/2$, then
$|\mathcal A|\geq b(a-r)>ab/2>D$, because
$b-2(t+1)=(a-1)t>0$.
It remains to consider $D+1-r<b$ and $r\geq a/2$.
Let $k=a-r$. Then $0\leq k\leq t$ and $y=at+k+1$.
Substituting into $|\mathcal A|\geq y(a-r)+(b-y)r$ gives
$|\mathcal A|\geq(at+k+1)k+(t+1-k)(a-k)
=D+k((a-1)t-a+2k)$.
If $k\geq1$, the expression in parentheses is at least
$2k-1\geq1$, contrary to $|\mathcal A|\leq D$.
Hence $k=0$, so $C$ contains all $a$ row vertices.
Its number $m$ of column vertices satisfies $m\leq D+1-a=at+1$.
Every pair $(i,j)$ with $y_j\notin V(C)$ belongs to $\mathcal A$;
there are $a(b-m)$ such pairs.
Therefore $|\mathcal A|\geq a(b-m)\geq a(t+1)=D$.
Since $|\mathcal A|\leq D$, equality holds throughout, giving
$|\mathcal A|=D$ and $m=at+1$.

Since each fibre has at most $D+1$ vertices,
$W_{\pi(o)}=\{o\}\cup\{z_{ij}:(i,j)\in\mathcal A\}$ and
$W_{\pi(x_1)}=V(C)$, both of cardinality $D+1$.
The edge $ox_1$ implies that $\pi(o)$ and $\pi(x_1)$ are distinct
and adjacent in $\mathcal P_\ell$.
Every vertex of $C$ is a row or column vertex, so the embedding
requires $V(C)\subseteq N_{G_\ell}(o)$.
However, the deleted perfect matching between $W_{\pi(o)}$ and
$W_{\pi(x_1)}$ gives
$|N_{G_\ell}(o)\cap W_{\pi(x_1)}|=D<D+1=|V(C)|$, a contradiction. Thus the construction does not contain
$F_{a,b}$ as a subgraph for any choice of the fibres or of the
deleted matchings.

Take $x=0$ for Theorem~\ref{thm:general}, and any fixed
$x\in[0,1]$ for Theorem~\ref{thm:exceptional}.
Write $N=\ell^2+\ell+1$ and choose uniformly at random a set
$S\subseteq V(\mathcal P_\ell)$ of size $h=\lfloor xN\rfloor$.
Give $W_p$ size $D+1$ exactly when $p\in S$.
The case $x=0$ gives the uniform construction.
For an edge $pq$ of $\mathcal P_\ell$, the number of edges
between its fibres is $D^2$ if neither endpoint is in $S$,
and $D^2+D$ otherwise. In particular, if both endpoints are
in $S$, deleting the perfect matching leaves
$(D+1)^2-(D+1)=D^2+D$ edges.
The probability that $S\cap\{p,q\}\ne\varnothing$ is
$1-(N-h)(N-h-1)/[N(N-1)]=2x-x^2+O(1/N)$.
By linearity of expectation, summing the contributions of the
edges of $\mathcal P_\ell$ shows that some choice of $S$ gives at least
$(D^2+D(2x-x^2+O(1/N)))e(\mathcal P_\ell)
=(D(D+2x-x^2)/2+o(1))N^{3/2}$ edges.
For every choice of $S$, the graph $G_\ell$ has
$n_\ell=DN+h=(D+x)N+O(1)$ vertices. It does not contain $F_{a,b}$
as a subgraph when $x=0$, or when $a<b$ and $b\equiv2\pmod{a+1}$.

Define $L_D(x)=D(D+2x-x^2)/[2(D+x)^{3/2}]$ for $0\leq x\leq1$.
For the choice of $S$ above,
$e(G_\ell)/n_\ell^{3/2}\geq L_D(x)-o(1)$ as $\ell\to\infty$.
For each sufficiently large integer $n$, the
prime number theorem supplies a prime
$\ell=\ell(n)<\sqrt{n/(D+x)}-2$ such that
$\lim_{n\to\infty}\ell(n)/\sqrt{n/(D+x)}=1$.
Since $h\leq xN$, we have
$n_\ell\leq(D+x)(\ell^2+\ell+1)\leq n$ and
$n_\ell/n\to1$. Adding $n-n_\ell$ isolated vertices cannot
create the connected graph $F_{a,b}$. We therefore obtain
$\ex(n,F_{a,b})\geq(L_D(x)-o(1))n^{3/2}$ for all sufficiently
large $n$. Taking $x=0$ gives $L_D(0)=\sqrt D/2$, proving the
lower bound in Theorem~\ref{thm:general}.

For Theorem~\ref{thm:exceptional}, take $x=1/4$ and write
$L=L_D(1/4)=D(16D+7)/[4(4D+1)^{3/2}]$. Since $D\geq4$,
we have $D^{-1}-4D^{-2}\geq0$ and $4+D^{-1}\leq17/4$. Hence
\[
 4L^2-D
 =\frac{32+D^{-1}-4D^{-2}}{4(4+D^{-1})^3}
 \geq\frac8{(17/4)^3}>\frac1{10}.
\]
Thus $L>\sqrt{D+1/10}/2$, proving the
stated lower bound in Theorem~\ref{thm:exceptional}.

For the parameters in Theorem~\ref{thm:exceptional}, the construction
gives a stronger lower bound by maximizing $L_D(x)$. Differentiation gives
$L_D'(x)=D(D-(4D+2)x-x^2)/[4(D+x)^{5/2}]$.
The polynomial $D-(4D+2)x-x^2$ is strictly decreasing for
$x\geq0$, equals $D>0$ at $x=0$, and is negative at $x=1/4$.
Its unique positive zero is
$x_D=\sqrt{4D^2+5D+1}-2D-1\in(0,1/4)$.
Hence the maximum on $[0,1]$ is $L_D(x_D)$, and the stronger
lower bound is $(L_D(x_D)-o(1))n^{3/2}$.
\end{proof}

\section{Proofs of the main theorems}

\subsection{Codegree estimates}

We begin with Ramsey results used in
the first lemma for the upper bounds. Ramsey \cite{Ramsey} proved that for all
positive integers $r,s$, every sufficiently large graph contains
either $K_r$ or an independent set of size $s$. The bipartite Ramsey theorem for two colours (see Beineke and Schwenk
\cite{BeinekeSchwenk}) states that for every positive integer $t$ there
is an integer $B(t)$ such that every map
$\chi:A\times B\to\{0,1\}$ with $|A|,|B|\geq B(t)$ is constant on
$A'\times B'$ for some $A'\subseteq A$ and $B'\subseteq B$ of size $t$.
We also need the following bipartite canonization theorem of Rado
\cite{Rado}, in the formulation of Dob\'ak and Mulrenin
\cite{DobakMulrenin}. Since the domain is finite,
the colours can be relabelled by integers, so the theorem applies
to an arbitrary colour set $\Omega$.

\begin{lemma}\label{lem:canonical-ramsey}
For every positive integer $t$ there is an integer $C(t)$ with the
following property. If $|A|,|B|\geq C(t)$ and
$c:A\times B\to\Omega$ is any map, then there are subsets
$A'\subseteq A$, $B'\subseteq B$, each of size $t$, such that on
$A'\times B'$ at least one of the following holds:
\begin{enumerate}
\item $c$ is constant;
\item $c$ is injective;
\item $c(u,w)=\phi(u)$ for an injective map $\phi:A'\to\Omega$;
\item $c(u,w)=\psi(w)$ for an injective map $\psi:B'\to\Omega$.
\end{enumerate}
\end{lemma}

For the rest of this section, fix integers $2\leq a\leq b$ and write
$X=a+b-1$. For distinct vertices $u,w$ of a graph $G$, write
\[
 Q(u,w)=N_G(u)\cap N_G(w),\qquad q(u,w)=|Q(u,w)|.
\]
Thus $Q(u,w)$ is their common neighbourhood and $q(u,w)$ their
\emph{codegree}. An oriented edge $(v,u)$ means $vu\in E(G)$.

The auxiliary graph in Lemma~\ref{lem:counting}(i) follows
Gao et al.\ \cite[Claim~4.4]{GeometricShapes}.
Part~(ii) combines the random common neighbourhood estimate in
Fox and Sudakov \cite[proof of Lemma~2.1]{FoxSudakovDRC} with a variant
of the random selection and colour exclusion argument in Dob\'ak and
Mulrenin \cite[proof of Theorem~1.3, Section~3]{DobakMulrenin}.

\begin{lemma}[Codegree estimates]\label{lem:counting}
There is a constant $C_1=C_1(a,b)$ such that every graph $G$ on
$n$ vertices that does not contain $F_{a,b}$ as a subgraph satisfies:
\begin{enumerate}
\item[(i)] $\sum_vd(v)^2=O_{a,b}(n^2)$,
$\sum_{\{u,w\}:q(u,w)>X}q(u,w)=o(n^2)$, and
$e(G)\leq\left(\frac{\sqrt X}{2}+o(1)\right)n^{3/2}$.
\item[(ii)] For a vertex $v$ with $d=d(v)>0$ and $0<\eta<1$,
let $L$ be the graph on $N(v)$ in which $xy\in E(L)$ exactly when
there is a vertex $z\ne v$ with
$\{x,y\}\subseteq Q(v,z)$ and $q(v,z)\leq\eta d$. Then
\[
 \left(\frac{2e(L)}{d^2}\right)^a
 \leq C_1\left(\eta+\frac1d\right).
\]
\end{enumerate}
\end{lemma}

\begin{proof}
For (i), for each $v\in V(G)$ define
$H_v$ on $N(v)$ by $uw\in E(H_v)$ if and only if
$q(u,w)>X$. An edge of $H_v$ need not belong to $G$.

\par\medskip\noindent\textbf{Claim.}\enspace There is an integer $r_0$,
depending only on $a,b$, such that each $H_v$ does not contain
$K_{r_0,r_0}$ as a subgraph.

\begingroup\renewcommand{\qedsymbol}{}
\begin{proof}[Proof of Claim]
Let $m=ab+1$ and $f=(a+1)(b+1)$.
In the notation for $F_{a,b}$ from Section~2, the row vertices
are $x_1,\ldots,x_a$, and the column vertices are
$y_1,\ldots,y_b$. Together they form one bipartition class
of size $a+b$; the centre $o$ and the $ab$ cell vertices
$z_{ij}$ form the other class, of size $ab+1$. Thus
$F_{a,b}\subseteq K_{a+b,ab+1}$; also $ab+1\geq a+b$.

Choose $T\geq m$, then $s_0$, then $r_0$ sufficiently large
for the $X$ successive applications of Lemma~\ref{lem:canonical-ramsey},
the bipartite Ramsey theorem and the ordinary Ramsey theorem
below, respectively. These integers depend only on $a,b$.

Suppose that $H_v$ contains a $K_{r_0,r_0}$ with parts $A_0,B_0$.
Since $K_f$ contains $F_{a,b}$, the graph $G$ does not contain
$K_f$ as a subgraph.
Ordinary Ramsey applied to $G[A_0]$ and $G[B_0]$ gives independent
sets $A_1\subseteq A_0$ and $B_1\subseteq B_0$ of size $s_0$.
Define $\chi:A_1\times B_1\to\{0,1\}$ by
$\chi(u,w)=1$ if $uw\in E(G)$ and $\chi(u,w)=0$ otherwise.
The bipartite Ramsey theorem gives sets
$L\subseteq A_1$ and $R\subseteq B_1$, each of size $T$, such that $\chi$ is constant
on $L\times R$. If $\chi=1$ on $L\times R$, then $G$ contains
$K_{a+b,ab+1}\supseteq F_{a,b}$, since
$|L|=|R|=T\geq m\geq ab+1$. Hence $\chi=0$ on $L\times R$.
Together with the independence of $L$ and $R$, this implies
that $G[L\cup R]$ is edgeless.

For every $(u,w)\in L\times R$, we have $uw\in E(H_v)$.
Thus $u,w$ have at least $X+1$ common neighbours in $G$, including $v$. Choose distinct
$c_1(u,w),\ldots,c_X(u,w)\in Q(u,w)\setminus\{v\}$.
These vertices lie outside $L\cup R$, since $G[L\cup R]$ is
edgeless. Apply Lemma~\ref{lem:canonical-ramsey} successively to the
$X$ maps $c_h:L\times R\to V(G)$. By the choice of $T$, there
are subsets of $L,R$, each of size $m$, on whose product every $c_h$
satisfies one of the four conclusions of the lemma; each conclusion
is preserved when the two sets are further restricted. Write $L,R$
also for these subsets. For each $(u,w)$ in their product,
$c_1(u,w),\ldots,c_X(u,w)$ are still distinct.

Consider the four possibilities for the maps $c_h$. If some
$c_h$ is injective on $L\times R$, choose distinct
$u_1,\ldots,u_a\in L$ and $w_1,\ldots,w_b\in R$. The vertices
$v,u_i,w_j,c_h(u_i,w_j)$ are all distinct and give the central,
row, column, and cell vertices of $F_{a,b}$, respectively.
The centre is adjacent to each $u_i,w_j$, and each
$c_h(u_i,w_j)$ is adjacent to $u_i,w_j$, giving all required
edges.

If $c_h(u,w)=\phi(u)$ with $\phi$
injective, the images of $X$ distinct vertices of $L$, together
with $v$, are $X+1=a+b$ distinct vertices adjacent to every
vertex of $R$. These images lie outside $L\cup R\cup\{v\}$.
Since $|R|=m\geq ab+1$, this gives
$K_{a+b,ab+1}\supseteq F_{a,b}$. If
$c_h(u,w)=\psi(w)$ with $\psi$ injective, the images of $X$
distinct vertices of $R$, together with $v$, are adjacent to
every vertex of $L$, giving $K_{a+b,ab+1}$ in $G$.
If every $c_h$ is constant,
write its value as $z_h$. The $z_h$ are distinct, and
$v,z_1,\ldots,z_X$ are adjacent to every vertex of $L\cup R$.
Again $G$ contains $K_{a+b,ab+1}$. Each case gives an
$F_{a,b}$ subgraph of $G$, contrary to the assumption that
$G$ does not contain $F_{a,b}$ as a subgraph. This proves the claim.
\end{proof}
\endgroup

By the claim and the K\H{o}v\'ari--S\'os--Tur\'an theorem
\cite{KST}, there is a constant
$C=C(a,b)$ with $e(H_v)\leq C d(v)^{2-1/r_0}$. For a fixed pair
$\{u,w\}$ with $q(u,w)>X$, the condition $uw\in E(H_v)$ holds
exactly when $v\in Q(u,w)$, for $q(u,w)$ choices of $v$. Thus
\[
 \sum_{\{u,w\}:q(u,w)>X}q(u,w)
 =\sum_v e(H_v)\leq C\sum_vd(v)^{2-1/r_0}.
\]
Define $M_2=\sum_vd(v)^2$ and $\alpha=1-1/(2r_0)\in(0,1)$.
Counting pairs of neighbours in two ways and then applying H\"older's
inequality gives
\[
 \frac{M_2-2e(G)}2=\sum_{\{u,w\}}q(u,w)
 \leq X\binom n2+C n^{1-\alpha}M_2^\alpha.
\]
The sum of $q(u,w)$ over pairs with $q(u,w)\leq X$ is at most
$X\binom n2$. For pairs with $q(u,w)>X$, use
$2-1/r_0=2\alpha$ and H\"older's inequality to obtain
$\sum_vd(v)^{2\alpha}\leq n^{1-\alpha}M_2^\alpha$.
Since $0<\alpha<1$, Young's inequality gives
$C n^{1-\alpha}M_2^\alpha\leq M_2/4+O_{a,b}(n)$.
Substituting this into the displayed count and using $2e(G)\leq n^2$
yields $M_2=O_{a,b}(n^2)$. The displayed bound for
$\sum_{\{u,w\}:q(u,w)>X}q(u,w)$ is now
$O_{a,b}(n^{1+\alpha})
=O_{a,b}(n^{2-1/(2r_0)})=o(n^2)$, uniformly over all such graphs $G$.

Cauchy--Schwarz gives
$4e(G)^2/n\leq M_2=O(n^2)$, hence $e(G)=O(n^{3/2})$
and $2e(G)=o(n^2)$.
Since $\sum_{\{u,w\}}q(u,w)\leq X\binom n2+o(n^2)$ and
$2e(G)=o(n^2)$, the identity
$M_2-2e(G)=2\sum_{\{u,w\}}q(u,w)$ gives
$M_2\leq Xn^2+o(n^2)$. A second application of
$4e(G)^2/n\leq M_2$ gives
$e(G)\leq(\sqrt X/2+o(1))n^{3/2}$, completing (i).

For (ii), the assertion is immediate if $e(L)=0$. For every
$xy\in E(L)$, choose $c(xy)\ne v$ with
$\{x,y\}\subseteq Q(v,c(xy))$ and $q(v,c(xy))\leq\eta d$.
For every $x\in V(L)$ and $z\ne v$, this gives
\[
 |\{y\in N_L(x):c(xy)=z\}|\leq\eta d,
\]
since all such $y$ belong to $Q(v,z)$.

Write $\rho=2e(L)/d^2\leq1$. Sample $x_1,\ldots,x_a$ independently
and uniformly from $V(L)$, allowing repetitions, and put
$Y_0=\bigcap_{i=1}^aN_L(x_i)$. By convexity,
\[
 \mathbb E|Y_0|=d^{-a}\sum_y d_L(y)^a\geq\rho^a d.
\]
Let $Y\subseteq Y_0$ consist of those $y$ for which
$x_1,\ldots,x_a,c(x_1y),\ldots,c(x_ay)$ are pairwise distinct.
In particular, $Y=\varnothing$ whenever two sampled entries coincide.
We claim that
\[
 \mathbb E|Y|\geq\rho^a d-O_a(\eta d+1).
\]
The possible equalities are $x_i=x_j$, $c(x_i y)=x_j$,
and $c(x_i y)=c(x_j y)$, with distinct indices.
For each fixed pair of indices, either of the first two types
contributes at most $d^{-2}\sum_y d_L(y)=\rho\leq1$ to the
expected loss, after the restrictions on the other entries are omitted.
For the third type, put
$d_z(y)=|\{x\in N_L(y):c(xy)=z\}|$.
Since $d_z(y)\leq\eta d$ and $\sum_zd_z(y)=d_L(y)$, its
contribution for each fixed pair of indices is at most
\[
 d^{-2}\sum_y\sum_zd_z(y)^2
 \leq d^{-2}\eta d\sum_y d_L(y)\leq\eta d.
\]
The claim follows by the union bound over the $O(a^2)$ pairs.

Suppose $\rho^a d>C_1(\eta d+1)$ for a sufficiently large
constant $C_1=C_1(a,b)$, and fix a tuple attaining at least
the preceding lower bound on $|Y|$.
We select $b$ column vertices from $Y$ successively.
Suppose $j<b$ columns and their $aj$ cell vertices have already
been chosen, all distinct, and let $U$ be the set of these vertices.
Then $|U|=(a+1)j=O_{a,b}(1)$.
The next column $y$ must avoid $U$, and each cell vertex
$c(x_i y)$ must also avoid $U$.
The bound on the fibres of $c$ excludes at most
$|U|+a|U|\eta d=O_{a,b}(\eta d+1)$ choices.
All other distinctness conditions follow from the definition of $Y$,
from $c(x_i y)\ne v$, and from the fact that $G$ is simple:
$x_i\ne y$ and $c(x_i y)\notin\{x_i,y\}$.
For sufficiently large $C_1$, all $b$ columns can therefore be selected.
The vertices $v,x_i,y_j,c(x_i y_j)$ give a copy of $F_{a,b}$,
a contradiction. Hence $\rho^a\leq C_1(\eta+1/d)$, proving (ii).
\end{proof}

\subsection{A local reciprocal degree inequality}

The exclusion of vertex disjoint stars in Lemma~\ref{lem:two-cores}
adapts the argument excluding a matching in Zhao et al.\
\cite[Claim~3.1]{ZhaoEtAl}, using large common neighbourhoods
to complete the embedding.

\begin{lemma}[Two independent sets with large common neighbourhoods]\label{lem:two-cores}
Fix integers $2\leq a\leq b$. There is a constant
$C_2=C_2(a,b)$ such that the following holds.
Let $G$ be a graph that does not contain $F_{a,b}$ as a subgraph.
Suppose that $Q,R\subseteq V(G)$ are nonempty,
disjoint independent sets, each of size at most $a+b-1$, and that
$|N(Q)|,|N(R)|\geq C_2$.
Let $T$ be the bipartite graph with parts $Q,R$ whose edges are
the edges of $G$ between these sets. Every vertex $x\in Q$ with
$d_T(x)>0$ satisfies
\[
 \frac1{d_T(x)}\sum_{y\in N_T(x)}\frac1{d_T(y)}
 \geq\max\left\{\frac1{a(a+b-d_T(x))},
                   \frac1{a+b-\lceil d_T(x)/a\rceil}\right\}.
\]
The same assertion holds with $Q$ and $R$ interchanged.
\end{lemma}

\begin{proof}
For a positive integer $c$, a $c$-star is a copy of $K_{1,c}$
with one vertex designated as its centre and the other $c$ vertices
as its leaves, including when $c=1$. Fix $x\in Q$ with
$k=d_T(x)>0$. If $k\leq a$, then
$a+b-\lceil k/a\rceil=a+b-1$ and
$a(a+b-k)\geq ab\geq a+b-1$. Each $y\in N_T(x)$ has
$d_T(y)\leq|Q|\leq a+b-1$, so the average of their reciprocal
degrees is at least $1/(a+b-1)$ and both asserted bounds follow.
Henceforth assume $k>a$, and let $B=N_T(x)$ and $c=a+b-k$.
Then $|B|=k$ and $1\leq c\leq b$, since
$k\leq|R|\leq a+b-1$.

\par\medskip\noindent\textbf{Claim.}\enspace The bipartite graph between $B$ and
$Q\setminus\{x\}$ does not contain the disjoint union of $a$ copies of a
$c$-star, each centred in $B$, as a subgraph. The same assertion holds
with $Q,R$ interchanged.

\par\noindent\emph{Proof of Claim.} Suppose that such $a$ vertex disjoint
$c$-stars exist. Use their centres as $x_1,\ldots,x_a$ and $x$ as $o$.
Use the remaining $k-a=b-c$ vertices of $B$ as
$y_1,\ldots,y_{b-c}$, and choose distinct
$y_{b-c+1},\ldots,y_b\in N(Q)\setminus(Q\cup R)$.
For each $i\in[a]$, use the $c$ leaves of the star centred at $x_i$
as $z_{i,b-c+1},\ldots,z_{ib}$.
All these cell vertices are distinct and belong to $Q\setminus\{x\}$.
For $b-c<j\leq b$, the edge $x_i z_{ij}$ belongs to the chosen
star, and $y_j z_{ij}\in E(G)$ because $y_j\in N(Q)$.

For $i\in[a]$ and $1\leq j\leq b-c$, choose distinct
$z_{ij}\in N(R)\setminus(Q\cup R)$, different also from all chosen
column vertices. Since $x_i,y_j\in B\subseteq R$, both
$x_i z_{ij}$ and $y_j z_{ij}$ are edges of $G$.
These choices can be made greedily: at each step, $Q\cup R$
and the previously chosen vertices forbid only $O_{a,b}(1)$ vertices,
so the common neighbourhoods suffice when $C_2$ is sufficiently large. The central vertex $x$ is adjacent
to every row vertex and every column vertex in $B$, since they
belong to $B=N_T(x)$; it is adjacent to the other column vertices
because they belong to $N(Q)$. Thus the chosen vertices define an injective map from $F_{a,b}$ to $G$
that preserves adjacency, contrary to the hypothesis. The same
argument applies with $Q,R$ interchanged. This proves the claim.

List the numbers $|N_T(y)\setminus\{x\}|$, for $y\in B$,
in nonincreasing order as $d_1\geq\cdots\geq d_k$. If
$d_j\geq c(a-j+1)$ for every $1\leq j\leq a$, choose the
vertices with degrees $d_a,d_{a-1},\ldots,d_1$ successively as
centres of $c$-stars. Before choosing the $c$ leaves for the
vertex with degree $d_j$, at most $c(a-j)$ leaves have already
been chosen, so this vertex has at least $c$ neighbours not
previously chosen. This gives $a$ vertex disjoint $c$-stars,
contrary to the claim. Hence, for some $0\leq i\leq a-1$,
$d_{i+1}\leq c(a-i)-1$.
Thus $d_j\leq c(a-i)-1$ for $i+1\leq j\leq k$. Since
every vertex of $B$ is adjacent to $x$, its degree in $T$ is one
more than its degree into $Q\setminus\{x\}$. Hence
\[
 \sum_{y\in B}\frac1{d_T(y)}
 =\sum_{j=1}^k\frac1{d_j+1}
 \geq\frac{k-i}{c(a-i)}
 \geq\frac{k}{ac}.
\]
The last inequality is equivalent to $i(k-a)\geq0$, proving the
bound $k^{-1}\sum_{y\in B}1/d_T(y)\geq1/[a(a+b-k)]$.

For the second bound, write $h=a+b-\lceil k/a\rceil$. If
$k>a(a+b)/(a+1)$, then
$ak-(a-1)(a+b)-k/a
=\frac{a-1}{a}((a+1)k-a(a+b))>0$.
The integer $ak-(a-1)(a+b)$ is therefore at least
$\lceil k/a\rceil$. As
$h-a(a+b-k)=ak-(a-1)(a+b)-\lceil k/a\rceil$,
we have $a(a+b-k)\leq h$. Since
$1/[a(a+b-k)]\geq1/h$, the first bound also gives the
second when $k>a(a+b)/(a+1)$. Henceforth assume
$a<k\leq a(a+b)/(a+1)$.

If every vertex of $B$ has degree at
most $h$, then each $y\in B$ has $1/d_T(y)\geq1/h$, proving
the second bound. Otherwise choose $z\in B$ with
$\ell=d_T(z)>h$.
For the remainder of this proof, redefine $c=a+b-\ell$ and let
$C=N_T(z)\subseteq Q$. Write $k-1=am+r$, where $m\geq1$ and
$0\leq r<a$. Since $\lceil k/a\rceil=m+1$ and $\ell>h$,
the integer $\ell$ satisfies $\ell\geq a+b-m$, so
$c\leq m$; also $\ell\leq|Q|\leq a+b-1$ gives $c\geq1$.
Thus $k-1\geq ac$. Moreover,
$k\leq a(a+b)/(a+1)$ gives
$a+b\geq\lceil(a+1)k/a\rceil=k+m+1$ and hence
$h=a+b-m-1\geq k$.

Apply the claim with $Q,R$ interchanged and $z$ in place of $x$.
The condition $d_T(z)>a$ holds because $\ell>h\geq k>a$.
Under this application, $B=N_T(x)$ is replaced by $C=N_T(z)$,
and each star has $c=a+b-\ell$ leaves. Thus there are no $a$
vertex disjoint $c$-stars centred in $C$ with leaves in
$R\setminus\{z\}$. The vertex $x\in C$ is adjacent to all $k-1$
vertices of $B\setminus\{z\}$. Consequently the graph between
$C\setminus\{x\}$ and $B\setminus\{z\}$ has no collection
of $a-1$ vertex disjoint $c$-stars centred in $C\setminus\{x\}$.
Indeed, $a-1$ such stars would have $(a-1)c$ leaves in
$B\setminus\{z\}$. Since $k-1\geq ac$, at least $c$ vertices of
$B\setminus\{z\}$ would remain. Together with $x$ as centre,
any $c$ of these vertices would form a further $c$-star disjoint
from the first $a-1$, contradicting the claim.

Order the vertices $w\in C\setminus\{x\}$ by the values
$|N_T(w)\cap(B\setminus\{z\})|$, in nonincreasing order.
If the $j$th value were at least $c(a-j)$ for every
$1\leq j\leq a-1$, choose these $a-1$ vertices as centres in
reverse order. At the $j$th vertex in the original ordering,
at most $c(a-1-j)$ neighbours would already have been chosen
as leaves. There would therefore be at least $c$ neighbours
available, giving $a-1$ vertex disjoint $c$-stars.
Therefore for some $0\leq i\leq a-2$ all but the first $i$ of
these degrees are at most $c(a-1-i)-1$.
The first $i$ degrees are at most $k-1$. Moreover
$|Q\setminus C|\leq a+b-1-\ell=c-1$.
The sum of the degrees in $T$ over $B\setminus\{z\}$ counts
edges to $x$, to $C\setminus\{x\}$, and to $Q\setminus C$.
These three counts are at most $k-1$,
$i(k-1)+(\ell-1-i)(c(a-1-i)-1)$, and
$(k-1)(c-1)$, respectively. Since $\ell=a+b-c$, this gives
\[
 \sum_{y\in B\setminus\{z\}}d_T(y)
 \leq M_i(c)
 :=(k-1)(c+i)+(a+b-c-1-i)\bigl(c(a-1-i)-1\bigr).
\]
For fixed $c$, $M_i(c)$ is a quadratic polynomial in $i$
with positive quadratic coefficient $c$. Since
$0\leq i\leq a-2$, $M_i(c)$ is at most the larger of
$M_0(c)=(k-1)c+(a+b-c-1)((a-1)c-1)$ and
$M_{a-2}(c)=(k-1)(c+a-2)+(b-c+1)(c-1)$.
Both $M_0(c)$ and $M_{a-2}(c)$ increase for $1\leq c\leq m$: their
derivatives are $(k-1)+(a-1)(a+b-1-2c)+1$ and
$(k-1)+b+2-2c$, respectively. The first is positive because
$a+b-1-2c\geq k-m>0$; the second is at least
$(a-2)m+b+2>0$. It therefore suffices
to bound $M_0(m)$ and $M_{a-2}(m)$.

Since $a+b-m-1=h\geq k$ and $k-1=am+r$, we have
\[
 (k-2)h-M_0(m)
 =\bigl(k-1-(a-1)m\bigr)h-(k-1)m
 =(m+r)h-(k-1)m\geq rk+m>0.
\]
For $M_{a-2}(m)$, use $k-1-m>0$ and $h\geq k$ to obtain
\[
 \begin{aligned}
 (k-2)h-M_{a-2}(m)
 &=(k-1-m)h-(k-1)(m+a-2)+(a-2)(m-1)\\
 &\geq(k-1)^2-2(k-1)m+(3-a)(k-1)+(a-3)m-a+2.
 \end{aligned}
\]
As a function of $k-1\geq am$, the polynomial on the right has
derivative $2(k-1)-2m+3-a\geq a+1>0$. At $k-1=am$ it equals
$a(a-2)m(m-1)+(2a-3)m-(a-2)\geq a-1>0$.
Consequently $\sum_{y\in B\setminus\{z\}}d_T(y)\leq(k-2)h$.
For every $y\in B\setminus\{z\}$, the edge $xy$ gives
$d_T(y)\geq1$. Also $k\geq a+1\geq3$, so $(k-2)h>0$.
Cauchy--Schwarz therefore yields
\[
 \sum_{y\in B\setminus\{z\}}\frac1{d_T(y)}
 \geq\frac{(k-1)^2}{\sum_{y\in B\setminus\{z\}}d_T(y)}
 \geq\frac{(k-1)^2}{(k-2)h}
 >\frac{k}{h}.
\]
The strict inequality holds because $(k-1)^2>(k-2)k$.
Adding $1/d_T(z)>0$ and dividing by $k$ proves the second bound.
Interchanging $Q$ and $R$ proves the assertion on the other side.
\end{proof}

\subsection{Partitions of neighbourhoods}

The weighted estimate in Lemma~\ref{lem:profile-partitions} uses the same
cancellation of common neighbour counts by reciprocal weights as
Gao et al.\ \cite[Claim~4.4 and Lemma~5.15]{GeometricShapes}.

\begin{lemma}[Neighbourhood partitions]\label{lem:profile-partitions}
Fix integers $2\leq a\leq b$ and recall $X=a+b-1$.
For every $\varepsilon,t>0$, there is an integer
$n_0=n_0(a,b,\varepsilon,t)$
such that, for every $n\geq n_0$ and every graph $G$ on $n$ vertices
that does not contain
$F_{a,b}$ as a subgraph, there exist a
spanning subgraph $G_0\subseteq G$ and a family
$(\Pi_u)_{u\in V(G_0)}$ such that each $\Pi_u$ is a partition of
$N_{G_0}(u)$. For $v\in N_{G_0}(u)$, let $\Pi_u(v)$ denote the unique part of
$\Pi_u$ containing $v$. These choices satisfy:
\begin{enumerate}
\item $e(G)-e(G_0)\leq\varepsilon n^{3/2}$;
\item every part of every $\Pi_u$ has size at most $X$, and every
      nonsingleton part $B$ is independent in $G_0$ and satisfies
      $|N_{G_0}(B)|\geq t$;
\item the following bound holds:
      \[
      \sum_{v:d_{G_0}(v)>0}
      \sum_{u,w\in N_{G_0}(v)}
      \frac1{|\Pi_u(v)\cap\Pi_w(v)|}
      \leq(1+\varepsilon)n^2.
      \]
\end{enumerate}
\end{lemma}

\begin{proof}
By Lemma~\ref{lem:counting}(i), fix $C_0=C_0(a,b)$ such that
$\sum_vd_H(v)^2\leq C_0n^2$ and $e(H)\leq C_0n^{3/2}$ for every
graph $H$ on sufficiently many vertices that does not contain
$F_{a,b}$ as a subgraph.
Fix also a sufficiently large constant $C_3=C_3(a,b)$.
Choose positive constants and integers $J,K$ according to
\[
\begin{aligned}
 0<1/n&\ll\eta_J\ll\zeta_{J-1}\ll\eta_{J-1}
          \ll\cdots\ll\eta_1\\
       &\ll\zeta_0\ll\eta_0\ll J^{-1}\ll K^{-1}\ll\varepsilon.
\end{aligned}
\]
The hierarchy is read from right to left: each parameter is chosen
sufficiently small in terms of those already fixed and $a,b$, and
$n$ is then taken sufficiently large.

Delete from $G$ every edge incident with a vertex of degree greater
than $K\sqrt n$, obtaining $G'$. At most
$C_0n^{3/2}/K<\varepsilon n^{3/2}/2$ edges are deleted, and
$\Delta(G')\leq K\sqrt n$.
Delete from $G'$ each edge $uw$ for which
$|N_{G'}(u)\cap N_{G'}(w)|\geq\eta_J\sqrt n$, obtaining $G_0$.
For large $n$, $\eta_J\sqrt n>X$, so Lemma~\ref{lem:counting}(i)
shows that the number deleted in this step is at most
$o(n^2)/(\eta_J\sqrt n)=o(n^{3/2})$.
Unless a graph is specified, degrees, neighbourhoods and codegrees
in the rest of this proof are taken in $G_0$. In particular,
$\Delta(G_0)\leq K\sqrt n$, and every
adjacent pair in $G_0$ has codegree less than $\eta_J\sqrt n$.

For $0\leq i\leq J$, let $H_i$ be the graph on $V(G_0)$ in which
distinct vertices $v,z$ are adjacent when $q(v,z)\geq\eta_i\sqrt n$.

\par\medskip\noindent\textbf{Claim 1.}\enspace For all sufficiently
large $n$, $d_{H_i}(v)<2XK/\eta_i$ for all $0\leq i\leq J$
and $v\in V(G_0)$.

\begingroup\renewcommand{\qedsymbol}{}
\begin{proof}[Proof of Claim 1]
Fix $i,v$ and let $R=d_{H_i}(v)$. Suppose $R\geq2XK/\eta_i$.
For each $y\in N(v)$, let
$t_y=|N_{H_i}(v)\cap N(y)|$.
Counting the pairs $(y,z)\in N(v)\times N_{H_i}(v)$ with $yz\in E(G_0)$ gives
$\sum_{y\in N(v)}t_y=\sum_{z\in N_{H_i}(v)}q(v,z)
 \geq R\eta_i\sqrt n$.
Also, for every subset of $N_{H_i}(v)$ of size $X$, its vertices together
with $v$ have at most $ab$ common neighbours: otherwise
$G_0$ contains $K_{a+b,ab+1}\supseteq F_{a,b}$.
Counting pairs consisting of $y\in N(v)$ and a subset of $N_{H_i}(v)\cap N(y)$ of size $X$ therefore gives
\[
 \sum_{y\in N(v)}\binom{t_y}{X}\leq ab\binom RX.
\]
Let $\mu=d(v)^{-1}\sum_{y\in N(v)}t_y$.
Since $d(v)\leq K\sqrt n$, we have
$\mu\geq R\eta_i/K\geq2X$. For every nonnegative integer $h$,
$\binom hX\geq ((h-X+1)_+)^X/X!$, where $s_+=\max\{s,0\}$.
Convexity and $\mu-X+1\geq\mu/2$ therefore give
\[
 \sum_{y\in N(v)}\binom{t_y}{X}
 \geq\frac{d(v)}{X!}\left(\frac{\mu}{2}\right)^X
 \geq\frac{d(v)}{X!}
       \left(\frac{R\eta_i}{2K}\right)^X.
\]
Combining the last two bounds and $\binom RX\leq R^X/X!$
yields $d(v)\leq ab(2K/\eta_i)^X$, a constant independent of
$n$. Yet $R>0$ implies $d(v)\geq q(v,z)\geq\eta_i\sqrt n$ for
some $z\in N_{H_i}(v)$, a contradiction when $n$ is large.
This proves Claim 1.
\end{proof}
\endgroup

For an oriented edge $(v,u)$ and $0\leq i\leq J$, define
$Q_i(v,u)=\{v\}\cup(N(u)\cap N_{H_i}(v))$.
These sets are nested in $i$.
Let $\mathcal E_{\mathrm{big}}$ consist of the oriented edges
$(v,u)$ with $|Q_J(v,u)|>X$.

\par\medskip\noindent\textbf{Claim 2.}\enspace
$|\mathcal E_{\mathrm{big}}|=O(n)$.

\begingroup\renewcommand{\qedsymbol}{}
\begin{proof}[Proof of Claim 2]
Fix $v$. If $|Q_J(v,u)|>X$, then $u$ is adjacent to $v$ and
to at least $X$ distinct vertices of $N_{H_J}(v)$.
By Claim 1, there are $O_{a,b,K,\eta_J}(1)$ choices of $X$
vertices from $N_{H_J}(v)$. Together with $v$, each such choice
has at most $ab$ common neighbours; otherwise $G_0$ contains
$K_{a+b,ab+1}\supseteq F_{a,b}$.
Thus there are $O(1)$ choices of $u$ for each $v$, proving the claim.
\end{proof}
\endgroup

For $(v,u)\notin\mathcal E_{\mathrm{big}}$,
$Q_0(v,u)$ contains $v$, and each strict inclusion
$Q_i(v,u)\subsetneq Q_{i+1}(v,u)$ increases cardinality by at
least one. Since $|Q_J(v,u)|\leq X$, there are at most $X-1$
indices $i\in\{0,\ldots,J-1\}$ with strict inclusion.
Summing over the $2e(G_0)$ oriented edges and averaging over
the $J$ values of $i$, fix $0\leq j<J$ such that at most
$2(X-1)e(G_0)/J$ oriented edges outside
$\mathcal E_{\mathrm{big}}$ satisfy
$Q_j(v,u)\ne Q_{j+1}(v,u)$.

Fix this $j$ for the rest of the proof. For $v\in V(G_0)$
and $u\in N(v)$, define
$P_v(u)=\{y\in N(v):Q_j(v,y)=Q_j(v,u)\}$.
For fixed $v$, the distinct sets $P_v(u)$ partition $N(v)$.
By Claim 1, there are $O_{a,b,K,\eta_j}(1)$ possible values of
$Q_j(v,u)$, and hence as many parts $P_v(u)$ for each $v$.
By the choice of $\zeta_j$, the parts of size less than
$\zeta_j\sqrt n$ contain at most $\sqrt n/J$ vertices in total.
Thus at most $n^{3/2}/J$ oriented edges $(v,u)$ satisfy
$|P_v(u)|<\zeta_j\sqrt n$.

Define $\mathcal E_*$ to consist of the oriented edges $(v,u)$ for which
$(z,u)\notin\mathcal E_{\mathrm{big}}$,
$Q_j(z,u)=Q_{j+1}(z,u)$ and $|P_z(u)|\geq\zeta_j\sqrt n$
hold for every $z\in Q_j(v,u)$.

\par\medskip\noindent\textbf{Claim 3.}\enspace At most
$(C_3/J)n^{3/2}+O(n)$ oriented edges lie outside
$\mathcal E_*$. Moreover, if $(v,u)\in\mathcal E_*$, then
$Q_j(z,u)=Q_j(v,u)$ for every $z\in Q_j(v,u)$.

\begingroup\renewcommand{\qedsymbol}{}
\begin{proof}[Proof of Claim 3]
For oriented edges
$(v,u)$ and $(z,u)$, symmetry of $H_j$ gives
$z\in Q_j(v,u)$ if and only if $v\in Q_j(z,u)$; this includes
$z=v$. Thus, for fixed $(z,u)$, exactly $|Q_j(z,u)|$
oriented edges $(v,u)$ satisfy $z\in Q_j(v,u)$.
This number is at most $X$ if
$(z,u)\notin\mathcal E_{\mathrm{big}}$; in all cases it is at
most $1+d_{H_j}(z)=O(1)$ by Claim 1.
For every $(v,u)\notin\mathcal E_*$, there is
$z\in Q_j(v,u)$ such that $(z,u)\in\mathcal E_{\mathrm{big}}$,
or $Q_j(z,u)\ne Q_{j+1}(z,u)$, or
$|P_z(u)|<\zeta_j\sqrt n$. Outside $\mathcal E_{\mathrm{big}}$, at most
$2(X-1)e(G_0)/J+n^{3/2}/J$ oriented edges $(z,u)$ satisfy
the second or third condition. Each accounts for at most $X$
choices of $(v,u)$, while edges in $\mathcal E_{\mathrm{big}}$
contribute $O(n)$ choices by Claims 1 and 2.
Using $e(G_0)\leq C_0n^{3/2}$, the total is at most
$(C_3/J)n^{3/2}+O(n)$ when $C_3=C_3(a,b)$ is
sufficiently large.

Fix $(v,u)\in\mathcal E_*$, write $Q=Q_j(v,u)$, and take $z\in Q$.
Since $v,z\in Q$, the definition of $\mathcal E_*$ gives
$|P_p(u)|\geq\zeta_j\sqrt n$ and
$Q_j(p,u)=Q_{j+1}(p,u)$ for each $p\in\{v,z\}$. Every $y\in P_v(u)$ is adjacent
to every vertex of $Q$, since
$Q_j(v,y)=Q\subseteq N(y)$. Therefore any two distinct
vertices of $Q$ have at least $\zeta_j\sqrt n$
common neighbours. For every $p\in Q\setminus\{z\}$, we have
$q(p,z)\geq\zeta_j\sqrt n\geq\eta_{j+1}\sqrt n$ and $p\in N(u)$,
so $p\in Q_{j+1}(z,u)=Q_j(z,u)$. Since $z\in Q_j(z,u)$ as well,
we obtain $Q\subseteq Q_j(z,u)$.
To prove $Q_j(z,u)\subseteq Q$, first note that
$v\in Q_j(z,u)$ because $z\in Q_j(v,u)$ and $H_j$ is undirected.
Every $y\in P_z(u)$ is adjacent to every vertex of
$Q_j(z,u)$, since
$Q_j(z,y)=Q_j(z,u)\subseteq N(y)$.
Thus each vertex of $Q_j(z,u)$ other than $v$ has at least
$\zeta_j\sqrt n$ common neighbours with $v$, and belongs to
$Q_{j+1}(v,u)=Q$. This proves the reverse inclusion and
hence $Q_j(z,u)=Q$ for every $z\in Q$, completing Claim 3.
\end{proof}
\endgroup

We use Claim 3 to construct the partitions. For a fixed $u$,
if $(v,u),(w,u)\in\mathcal E_*$ and
$z\in Q_j(v,u)\cap Q_j(w,u)$, Claim 3 gives
$Q_j(v,u)=Q_j(z,u)=Q_j(w,u)$.
Thus any two distinct sets among
$\{Q_j(v,u):(v,u)\in\mathcal E_*\}$ are disjoint.
Use these sets as parts of $\Pi_u$, and use $\{w\}$ as a part
for each $w\in N(u)$ not belonging to any of them. For every
$(v,u)\in\mathcal E_*$, this gives
$\Pi_u(v)=Q_j(v,u)$.
Each such set $Q=Q_j(v,u)$ has $|Q|\leq X$, because
$(v,u)\notin\mathcal E_{\mathrm{big}}$ (take $z=v$ in the
definition of $\mathcal E_*$). For every $y\in P_v(u)$,
$Q_j(v,y)=Q\subseteq N(y)$, so $P_v(u)\subseteq N(Q)$ and
$|N(Q)|\geq\zeta_j\sqrt n$.
Every two distinct vertices of $Q$ have at least
$\zeta_j\sqrt n$
common neighbours, as shown in the proof of Claim 3.
The hierarchy gives $\eta_J\leq\eta_{j+1}<\zeta_j$.
For any distinct $p,r\in Q$, we therefore have
$q(p,r)\geq\zeta_j\sqrt n>\eta_J\sqrt n$.
The deletion from $G'$ to $G_0$ ensures that such $p,r$ are
not adjacent in $G_0$; hence $Q$ is independent.
For sufficiently large $n$, $\zeta_j\sqrt n\geq t$.
The singleton parts have size at most $X$, so part (2) follows.

It remains to prove part (3). Its sum is indexed by ordered
triples $(u,v,w)$ with $u,w\in N(v)$, including $u=w$.
If $u\ne w$, then $q(u,w)>0$ because $v\in Q(u,w)$.
We compare $1/|\Pi_u(v)\cap\Pi_w(v)|$ with $1/q(u,w)$.
The two values are equal when $(v,u),(v,w)\in\mathcal E_*$
and every $z\in Q(u,w)$ other than $v$ satisfies
$q(v,z)\geq\eta_j\sqrt n$. Indeed, under these conditions
\[
 \Pi_u(v)\cap\Pi_w(v)=Q(u,w).
\]
The left side is contained in the right side. Conversely,
$v$ belongs to both parts; any other common neighbour $z$
has $vz\in E(H_j)$ and so belongs to both
$Q_j(v,u)=\Pi_u(v)$ and $Q_j(v,w)=\Pi_w(v)$.
Consequently, if $(v,u),(v,w)\in\mathcal E_*$ and
$|\Pi_u(v)\cap\Pi_w(v)|\ne q(u,w)$, then there is a vertex $z\ne v$ with
$u,w\in Q(v,z)$ and
$q(v,z)<\eta_j\sqrt n\leq\eta_0\sqrt n$.

We first sum $1/q(u,w)$ over all ordered triples
$(u,v,w)$ with $u,w\in N(v)$ and $u\ne w$. Interchanging
the finite sums gives
\[
 \sum_v\sum_{\substack{u,w\in N(v)\\u\ne w}}\frac1{q(u,w)}
 =\sum_{\substack{u\ne w\\q(u,w)>0}}
   \ \sum_{v\in Q(u,w)}\frac1{q(u,w)}
 =|\{(u,w):u\ne w,\ q(u,w)>0\}|
 \leq n(n-1).
\]
First, Claim 3 bounds the number of oriented edges outside
$\mathcal E_*$ by $(C_3/J)n^{3/2}+O(n)$.
Each such edge occurs in at most $2K\sqrt n$ ordered triples
$(u,v,w)$ with $u\ne w$. Hence at most
$(2KC_3/J)n^2+o(n^2)$ triples have
$(v,u)\notin\mathcal E_*$ or $(v,w)\notin\mathcal E_*$.

Second, consider triples with $u\ne w$, $u,w\in Q(v,z)$ and
$q(v,z)\leq\eta_0\sqrt n$ for some $z\ne v$.
The number of ordered triples $(u,v,w)$ with
$u,w\in N(v)$ and $d(v)<\sqrt{\eta_0n}$ is at most
$\sum_{d(v)<\sqrt{\eta_0n}}d(v)^2\leq\eta_0n^2$.
This also bounds the number satisfying the additional conditions above.
For $d(v)\geq\sqrt{\eta_0n}$,
$q(v,z)\leq\eta_0\sqrt n\leq\sqrt{\eta_0}\,d(v)$.
Let $L_v$ be the graph defined in Lemma~\ref{lem:counting}(ii)
with $\eta=\sqrt{\eta_0}<1$. Each unordered pair
$\{u,w\}\subseteq Q(v,z)$ considered here is an edge of $L_v$.
Lemma~\ref{lem:counting}(ii) bounds $2e(L_v)$, and hence
the number of ordered pairs $(u,w)$ at this $v$, by
\[
 \left[C_1\left(\sqrt{\eta_0}+\frac1{d(v)}\right)
 \right]^{1/a}d(v)^2.
\]
Since $d(v)\geq\sqrt{\eta_0n}$ in this sum, we have
$1/d(v)=o(1)$ uniformly. Summing and using
$\sum_vd(v)^2\leq C_0n^2$ therefore gives a total of at most
$[C_0C_1^{1/a}\eta_0^{1/(2a)}+o(1)]n^2$.

Finally, $v\in\Pi_u(v)\cap\Pi_w(v)$ gives
$1/|\Pi_u(v)\cap\Pi_w(v)|\leq1$.
For a triple $(u,v,w)$ with $u\ne w$ not counted in either of
the preceding two estimates, we have
$1/|\Pi_u(v)\cap\Pi_w(v)|=1/q(u,w)$.
For every remaining triple with $u\ne w$, the difference
$1/|\Pi_u(v)\cap\Pi_w(v)|-1/q(u,w)$ is at most one.
The terms with $u=w$ contribute at most
$\sum_vd(v)=2e(G_0)=O(n^{3/2})$.
Combining these estimates with the displayed sum over
ordered triples $(u,v,w)$ with $u\ne w$ gives
\[
 \sum_{v:d(v)>0}\sum_{u,w\in N(v)}
 \frac1{|\Pi_u(v)\cap\Pi_w(v)|}
 \leq\left[1+\frac{2KC_3}{J}+\eta_0
       +C_0C_1^{1/a}\eta_0^{1/(2a)}+o(1)\right]n^2.
\]
By the parameter hierarchy, this is at most $(1+\varepsilon)n^2$
for sufficiently large $n$. Taking $n_0$ large enough to meet all
preceding size requirements proves part (3). The deletions
from $G$ to $G'$ and then to $G_0$ together remove at most
$\varepsilon n^{3/2}$ edges for sufficiently large $n$,
proving part (1) as well.\qedhere
\end{proof}

\subsection{Completion of the upper bounds}

\begin{proof}[Proof of Theorems~\ref{thm:general} and
\ref{thm:exceptional}: upper bounds]
For integers $1\leq s\leq X$, define
$g(s)=\max\{1/[a(a+b-s)],1/[a+b-\lceil s/a\rceil]\}$.
Both denominators are positive integers, so $0<g(s)\leq1$.

Fix $\varepsilon>0$. Let $C_2$ be the constant in
Lemma~\ref{lem:two-cores}, and apply
Lemma~\ref{lem:profile-partitions} with $t=C_2$ to a graph $G$ on $n$
vertices that does not contain $F_{a,b}$ as a subgraph, and let
$G_0$ be the resulting spanning subgraph.
All degrees and neighbourhoods in this proof refer to $G_0$.
For each vertex $v$ with $d(v)>0$, define
$I_v=d(v)^{-2}\sum_{u,w\in N(v)}|\Pi_u(v)\cap\Pi_w(v)|^{-1}$,
using the partitions supplied by Lemma~\ref{lem:profile-partitions}.
For an oriented edge $(v,u)$, define
$s(v,u)=|\Pi_u(v)|$. By Lemma~\ref{lem:profile-partitions}(2),
$s(v,u)\in\{1,\ldots,X\}$.

Fix $v$ with $d(v)>0$, a part $B\in\Pi_v$, and $w\in N(v)$.
We claim that
$|B|^{-1}\sum_{z\in B}|\Pi_w(v)\cap\Pi_z(v)|^{-1}\geq g(|B|)$.
Every intersection $\Pi_w(v)\cap\Pi_z(v)$ contains $v$.
If $|B|=1$, then $g(1)=1/X$ and
$|\Pi_w(v)\cap\Pi_z(v)|\leq X$. If $|\Pi_w(v)|=1$, then
$\Pi_w(v)=\{v\}$, so every intersection has size $1$ and
$g(|B|)\leq1$.
Suppose now that $|B|,|\Pi_w(v)|>1$. By
Lemma~\ref{lem:profile-partitions}(2), both parts are independent,
have size at most $X$, and each has at least
$C_2$ common neighbours. We have $B\cap\Pi_w(v)=\varnothing$: every vertex of $B$ is
adjacent to $v$, whereas no vertex of the independent set
$\Pi_w(v)$ is adjacent to its member $v$. Let $T$ be the bipartite graph with parts $B$ and
$\Pi_w(v)$ whose edges are the edges of $G_0$ between these sets. The vertex
$v\in\Pi_w(v)$ has degree $|B|$ in $T$, so
Lemma~\ref{lem:two-cores}, with its parts interchanged, gives
$|B|^{-1}\sum_{z\in B}d_T(z)^{-1}\geq g(|B|)$.
For each $z\in B$, the inclusion
$\Pi_w(v)\cap\Pi_z(v)\subseteq\Pi_w(v)\cap N(z)$ gives
$|\Pi_w(v)\cap\Pi_z(v)|\leq d_T(z)$.
Taking reciprocals and averaging over $z\in B$ proves the claim.

For fixed $w\in N(v)$, multiply the claim by $|B|$ and sum
over all parts $B\in\Pi_v$.
Each $z\in N(v)$ belongs to exactly one part,
$\Pi_v(z)$, whose size is $s(z,v)$. Hence
$\sum_{z\in N(v)}|\Pi_w(v)\cap\Pi_z(v)|^{-1}
\geq\sum_{z\in N(v)}g(s(z,v))$.
Summing this inequality over the $d(v)$ choices of $w$ and
dividing by $d(v)^2$ gives
$I_v\geq d(v)^{-1}\sum_{u\in N(v)}g(s(u,v))$.

For $u,w\in N(v)$,
$|\Pi_u(v)\cap\Pi_w(v)|\leq\min\{s(v,u),s(v,w)\}
\leq\sqrt{s(v,u)s(v,w)}$.
Taking reciprocals and summing over $u,w\in N(v)$ gives
$I_v\geq d(v)^{-2}(\sum_{u\in N(v)}\frac1{\sqrt{s(v,u)}})^2$ and hence
$\sqrt{I_v}\geq d(v)^{-1}\sum_{u\in N(v)}\frac1{\sqrt{s(v,u)}}$.
The lower bound for $I_v$ obtained by summing the claim,
together with concavity of the square root, gives $\sqrt{I_v}\geq
d(v)^{-1}\sum_{u\in N(v)}\sqrt{g(s(u,v))}$.
Define $Z=\sum_{v:d(v)>0}d(v)\sqrt{I_v}$, and let
$E^{\mathrm{or}}(G_0)$ denote the set of oriented edges of $G_0$.
Multiplying each lower bound for $\sqrt{I_v}$ by $d(v)$ and
summing over $v$ gives
\[
Z\geq\sum_{(v,u)\in E^{\mathrm{or}}(G_0)}\frac1{\sqrt{s(v,u)}}
\quad\text{and}\quad
Z\geq\sum_{(v,u)\in E^{\mathrm{or}}(G_0)}\sqrt{g(s(u,v))}.
\]
Reversing the orientation in the second sum and averaging gives
\[
Z\geq\frac12\sum_{(v,u)\in E^{\mathrm{or}}(G_0)}
 \left(\frac1{\sqrt{s(v,u)}}+\sqrt{g(s(v,u))}\right).
\]
By Cauchy--Schwarz,
$Z^2\leq |\{v:d(v)>0\}|\sum_{d(v)>0}d(v)^2I_v$.
We have $|\{v:d(v)>0\}|\leq n$ and, by
Lemma~\ref{lem:profile-partitions}(3) and the definition of $I_v$,
$\sum_{d(v)>0}d(v)^2I_v\leq(1+\varepsilon)n^2$.
Thus $Z\leq\sqrt{1+\varepsilon}\,n^{3/2}$.

Define $M=\frac12\min_{1\leq s\leq X}
(\frac1{\sqrt{s}}+\sqrt{g(s)})>0$, where the minimum is over
integer $s$. Since
$|E^{\mathrm{or}}(G_0)|=2e(G_0)$, the two bounds for $Z$ give
$2M e(G_0)\leq\sqrt{1+\varepsilon}\,n^{3/2}$.
By Lemma~\ref{lem:profile-partitions}(1),
$e(G)\leq[\sqrt{1+\varepsilon}/(2M)+\varepsilon]n^{3/2}$.
The number $M$ depends only on $a,b$. Letting
$n\to\infty$ and then $\varepsilon\to0$ gives
$\limsup_{n\to\infty}\ex(n,F_{a,b})/n^{3/2}\leq1/(2M)$.

It remains to give an explicit lower bound for $M$.
For $1\leq s\leq X$, define
$h(s)=1/g(s)=\min\{a(a+b-s),a+b-\lceil s/a\rceil\}$.
Write $a+b=(a+1)t+\lambda$, where $t\geq1$ is an integer and
$0\leq\lambda\leq a$. By its definition, $D=at$ if
$\lambda=0$, and $D=at+\lambda-1$ if $\lambda\geq1$.
For $s\leq D$, the function
$s-\lceil s/a\rceil$ is nondecreasing on integers, because
its value increases by $0$ or $1$ when $s$ increases by $1$.
Since $h(s)\leq a+b-\lceil s/a\rceil$, we have
$s+h(s)\leq a+b+s-\lceil s/a\rceil
\leq a+b+D-\lceil D/a\rceil$.
$a+b+D-\lceil D/a\rceil$ equals $2D+1$ when $\lambda=1$ and
$2D$ when $\lambda=0$ or $2\leq\lambda\leq a$.
For $s\geq D+1$, the inequality $h(s)\leq a(a+b-s)$ gives
$s+h(s)\leq a(a+b)-(a-1)s
\leq a(a+b)-(a-1)(D+1)$.
$a(a+b)-(a-1)(D+1)$ equals $2D-a+1$ when $\lambda=0$,
$2D+1$ when $\lambda=1$, and $2D-\lambda+2$ when
$2\leq\lambda\leq a$. Combining the two ranges of $s$,
we obtain $s+h(s)\leq2D$ if $\lambda\ne1$, and
$s+h(s)\leq2D+1$ if $\lambda=1$.
Since $x\mapsto x^{-1/2}$ is convex and decreasing, when
$\lambda\ne1$ we obtain
\[
 \frac1{\sqrt{s}}+\frac1{\sqrt{h(s)}}\geq\frac2{\sqrt D}.
\]
Here $h(D)=D$, so equality is attained at $s=D$.
When $\lambda=1$, both $s$ and $h(s)$ are positive integers.
Among positive integers $p,q$ with $p+q\leq2D+1$, the minimum
of $p^{-1/2}+q^{-1/2}$ is attained at $p=D$, $q=D+1$.
Indeed, increasing $p$ or $q$ decreases this sum, so a minimizing
pair has $p+q=2D+1$. For a fixed sum, convexity of $x^{-1/2}$
implies that the minimum occurs when $|p-q|\leq1$. Hence
\[
 \frac1{\sqrt{s}}+\frac1{\sqrt{h(s)}}\geq
 \frac1{\sqrt D}+\frac1{\sqrt{D+1}}.
\]
Moreover, $h(D)=D+1$ and $h(D+1)=D$, and both $D$ and $D+1$
belong to $\{1,\ldots,X\}$. Thus the minimum defining $M$ is attained at $s=D$ or $s=D+1$, and
\[
 \frac1{2M}=
 \begin{cases}
  \sqrt D/2,&\lambda\ne1,\\[2pt]
  \displaystyle\frac{\sqrt{D(D+1)}}{\sqrt D+\sqrt{D+1}},&\lambda=1.
 \end{cases}
\]
Strict convexity gives
$1/\sqrt D+1/\sqrt{D+1}>2/\sqrt{D+1/2}$, and therefore
$1/(2M)<\sqrt{D+1/2}/2$ when $\lambda=1$.

When $\lambda=1$, the identity $a+b-1=(a+1)t$ gives
$D=at=a(a+b-1)/(a+1)$.
Since $a+b\equiv b-1\pmod{a+1}$, the condition $\lambda=1$
is equivalent to $b\equiv2\pmod{a+1}$.

Combining these upper bounds with the lower bounds in Section~2
proves Theorems~\ref{thm:general} and~\ref{thm:exceptional}.
\end{proof}

\section*{Declaration of generative AI and AI-assisted technologies in the manuscript preparation process}

During the preparation of this work, the authors used OpenAI Codex to assist with the development of some mathematical proofs. The authors reviewed the AI-assisted material and take full responsibility for the correctness of the proofs and the content of the article.
\end{document}